\documentclass{article}

\usepackage{amsthm,amsmath,amssymb}
\usepackage{mathrsfs}
\usepackage{graphicx}
\usepackage{color}
\usepackage{ulem}

\newtheorem{theorem}{Theorem}
\newtheorem{lemma}{Lemma}
\newtheorem{corollary}{Corollary}

\newtheorem{definition}{Definition}
\newtheorem{proposition}{Proposition}
\newtheorem{remark}{Remark}
\newtheorem{example}{Example}

\newcommand{\ls}[1]
    {\dimen0=\fontdimen6\the\font\lineskip=#1\dimen0
     \advance\lineskip.5\fontdimen5\the\font
     \advance\lineskip-\dimen0
     \lineskiplimit=0.9\lineskip
     \baselineskip=\lineskip
     \advance\baselineskip\dimen0
     \normallineskip\lineskip\normallineskiplimit\lineskiplimit
     \normalbaselineskip\baselineskip
     \ignorespaces}

\UseRawInputEncoding
\begin{document}

\bibliographystyle{abbrv}

\title{Order preserving and order reversing operators on the class of $L^0$-convex functions in complete random normed modules}
\author{Mingzhi Wu$^1$ \quad Tiexin Guo$^{2}$\quad Long Long$^{2}$\\
1. School of Mathematics and Physics, China University of Geosciences,\\ Wuhan {\rm 430074}, China \\
2. School of Mathematics and Statistics,
Central South University,\\ Changsha {\rm 410083}, China\\
Email: wumz@cug.edu.cn; tiexinguo@csu.edu.cn; longlong@csu.edu.cn
}

\date{}
 \maketitle

\thispagestyle{plain}
\setcounter{page}{1}

\begin{abstract}
Based on both the fundamental theorem of affine geometry in regular $L^0$-modules and the recent progress in random convex analysis, this paper characterizes the stable fully order preserving and order reversing operators acting on the class of proper lower semicontinuous $L^0$-convex functions in complete random normed modules.
\end{abstract}



\ls{1.5}

\section{Introduction and the main results}

In 2009, Artstein-Avidan and Milman \cite{AM} established the characterizations of the fully order preserving and the fully order reversing operators acting on the class of lower semicontinuous convex functions defined on the n-dimensional Euclidean space $\mathbb{R}^n$. In 2015, Iusem, Reem and Svaiter \cite{IRS} successfully generalized the work of \cite{AM} to Banach spaces in a nontrivial manner. In 2021, Cheng and Luo \cite{CL} solved several difficult open problems of fully order-preserving and order-reversing operators in the case of Banach spaces. The purpose of this paper is to generalize the work of \cite{IRS} to complete random normed modules, which are a random generalization of Banach spaces. Random functional analysis is just based on the idea of randomizing the traditional space theory of functional analysis. To introduce the main results of this paper, let us first recall some usual terminologies and notations frequently used in random functional analysis.
\par
Throughout this paper $(\Omega, \mathcal{F},P)$ always denotes a given probability space, $\mathbb{K}$ the scalar field $\mathbb{R}$ of real numbers or $\mathbb{C}$ of complex numbers, $L^0(\mathcal{F}, \mathbb{K})$ the algebra of equivalence classes of $\mathbb{K}$-valued random variables on $(\Omega, \mathcal{F},P)$ (as usual, two random variables equal almost surely (a.s.) are said to be equivalent), $L^0(\mathcal{F}):= L^0(\mathcal{F}, \mathbb{R})$ and $\bar{L}^0(\mathcal{F})$ the set of equivalence classes of extended real-valued random variables on $(\Omega, \mathcal{F},P)$.
\par
$\bar{L}^0(\mathcal{F})$ is usually partially ordered via $\xi\leq \eta$ if $\xi^0(\omega)\leq \eta^0(\omega)$ for almost all $\omega$ in $\Omega$, where $\xi^0$ and $\eta^0$ are arbitrarily chosen representatives of $\xi$ and $\eta$, respectively. The nice properties of $(\bar{L}^0(\mathcal{F}), \leq)$ are surveyed as follows.

\begin{proposition}\label{proposition1.1}(\cite{DS})
$(\bar{L}^0(\mathcal{F}), \leq)$ is a complete lattice, $\vee H$ and $\wedge H$ stand for the supremum and infimum of $H$, respectively, for any nonempty subset $H$ of $\bar{L}^0(\mathcal{F})$. Specially, the following statements are true:
\begin{enumerate}
\item There are two sequences $\{a_n, n\in \mathbb{N}\}$ and $\{b_n, n\in \mathbb{N}\}$ in $H$ for any nonempty subset $H$ of $\bar{L}^0(\mathcal{F})$ such that $\vee_{n\geq 1}a_n=\vee H$ and $\wedge_{n\geq 1}b_n= \wedge H$.
\item If $H$ in (1) is directed upwards (downwards), i.e., there exists $h_3\in H$ for any two elements $h_1$ and $h_2$ in $H$ such that $h_1\vee h_2\leq h_3$ ($h_3\leq h_1\wedge h_2$), then $\{a_n, n\in \mathbb{N}\}$ (correspondingly, $\{b_n, n\in \mathbb{N}\}$) can be chosen as nondecreasing (nonincreasing).
\item $(L^0(\mathcal{F}), \leq)$ is a Dedekind complete lattice, i.e., every nonempty subset with an upper bound has a supremum.
\end{enumerate}
\end{proposition}
\par
As usual, for any two elements $\xi$ and $\eta$ in $\bar{L}^0(\mathcal{F})$, $\xi< \eta$ means $\xi\leq \eta$ but $\xi\neq \eta$, whereas, for any $A\in \mathcal{F}$ with $P(A)>0$, $\xi< \eta$ on $A$ means $\xi^0(\omega)< \eta^0(\omega)$ for almost all $\omega$ in $A$, where $\xi^0$ and $\eta^0$ are arbitrarily chosen representatives of $\xi$ and $\eta$, respectively.
\par
In this paper, the following two notations are always employed:\\
$L^0_+(\mathcal{F})=\{\xi\in L^0(\mathcal{F}): \xi\geq 0\};$\\
$L^0_{++}(\mathcal{F})=\{\xi\in L^0(\mathcal{F}): \xi>0$ on $\Omega\}.$

\begin{definition}\label{definition1.2}(\cite{Guo-JFA})
Let $E$ be a left module over the algebra $L^0(\mathcal{F},\mathbb{K})$ (briefly, an $L^0(\mathcal{F},\mathbb{K})$-module). A mapping $\|\cdot\|$ from $E$ to $L^0_+(\mathcal{F})$ is an $L^0$-seminorm on $E$ if the following conditions are satisfied:
\begin{enumerate}
\item $\|\xi x\|=|\xi|\|x\|$ for all $\xi\in L^0(\mathcal{F},\mathbb{K})$ and $x\in E$;
\item $\|x+y\|\leq \|x\|+\|y\|$ for all $x$ and $y$ in $E$.
\end{enumerate}
If, in addition, $\|x\|=0$ implies $x=\theta$ (the null of $E$), then the $L^0$-seminorm is an $L^0$-norm on $E$, at this time $(E,\|\cdot\|)$ is a random normed module (briefly, an $RN$ module) over $\mathbb{K}$ with base $(\Omega, \mathcal{F}, P)$. If a family $\mathcal{P}$ of $L^0$-seminorms on $E$ satisfies that $\vee\{\|x\|: \|\cdot\|\in \mathcal{P}\}=0$ implies $x=\theta$, then the ordered pair $(E, \mathcal{P})$ is a random locally convex module (briefly, an $RLC$ module) over $\mathbb{K}$ with bases $(\Omega, \mathcal{F}, P)$.
\end{definition}

\par
When $(\Omega, \mathcal{F}, P)$ is trivial, namely $\mathcal{F}=\{\Omega, \emptyset\}$, an $RN$ or $RLC$ module over $\mathbb{K}$ with base $(\Omega, \mathcal{F}, P)$ just reduces to an ordinary normed or locally convex space over $\mathbb{K}$, and thus an $RN$ or $RLC$ module is a random generalization of an ordinary normed or locally convex space. The simplest $RN$ module is $L^0(\mathcal{F}, \mathbb{K})$ endowed with the $L^0$-norm $|\cdot|$ (namely the absolute value mapping). It is well known that $L^0(\mathcal{F}, \mathbb{K})$ is a metrizable linear topological space (in fact, also a topological algebra) in the topology of convergence in probability, the topology of convergence in probability is often nonlocally convex, for example, $L^0(\mathcal{F},\mathbb{K})$ has the trivial topological dual when $\mathcal{F}$ is atomless. The $(\varepsilon,\lambda)$-topology defined in Proposition \ref{proposition1.3} below is exactly a natural generalization of the topology of convergence in probability.

\begin{proposition}\label{proposition1.3}(\cite{Guo-survey1,Guo-JFA,GP})
Let $(E, \mathcal{P})$ be an $RLC$ module over $\mathbb{K}$ with base $(\Omega, \mathcal{F}, P)$. Denote by $\mathcal{P}_{f}$ the family of nonempty finite subsets of $\mathcal{P}$, for each $Q\in \mathcal{P}_{f}$, the $L^0$-seminorm $\|\cdot\|_Q$ is defined by $\|x\|_Q=\vee \{\|x\|: \|\cdot\| \in Q\}$ for any $x\in E$. For any given positive numbers $\varepsilon$ and $\lambda$ with $0<\lambda <1$, and for any $Q\in \mathcal{P}_{f}$, let $\mathcal{U}_{\theta}(Q,\varepsilon,\lambda)=\{x\in E: P\{\omega\in \Omega~|~ \|x\|_Q(\omega)< \varepsilon\}>1-\lambda\}$, then $\{\mathcal{U}_{\theta}(Q,\varepsilon,\lambda): Q\in \mathcal{P}_{f}, \varepsilon>0, 0<\lambda <1\}$ forms a local base of some Hausdorff linear topology for $E$, called the $(\varepsilon,\lambda)$-topology induced by $\mathcal{P}$, denoted by $\mathcal{T}_{\varepsilon,\lambda}$. Furthermore, the following statements hold:
\begin{enumerate}
  \item $(E, \mathcal{T}_{\varepsilon,\lambda})$ is a Hausdorff topological module over the topological algebra $L^0(\mathcal{F}, \mathbb{K})$.
  \item $\mathcal{T}_{\varepsilon,\lambda}$ is merizable when $(E, \mathcal{P})$ is an $RN$ module $(E, \|\cdot\|)$, namely $\mathcal{P}$ reduces to a single $L^0$-norm $\|\cdot\|$.
\end{enumerate}
\end{proposition}
\par
In this paper, we always assume that an $RLC$ module is endowed with its $(\varepsilon,\lambda)$-topology. It is easy to see that in an $RLC$ module $(E, \mathcal{P})$ over $\mathbb{K}$ with base $(\Omega, \mathcal{F}, P)$, a net $\{x_{\alpha}, \alpha\in \Gamma\}$ converges in the $(\varepsilon,\lambda)$-topology to $x$ iff $\{\|x_{\alpha}-x\|, \alpha\in \Gamma\}$ converges in probability to 0 for each $\|\cdot\|\in \mathcal{P}$, in particular in an $RN$ module $(E, \|\cdot\|)$ a sequence $\{x_n, n\in \mathbb{N}\}$ converges in the $(\varepsilon,\lambda)$-topology to $x$ iff $\{\|x_n-x\|, n\in \mathbb{N}\}$ converges in probability to 0.
\par
The $(\varepsilon,\lambda)$-topology is essentially a kind of nonlocally convex topology, the theory of traditional conjugate spaces universally fails for the development of random locally convex modules. The idea of random conjugate spaces is introduced and have been systematically developed in order to overcome the above-stated obstacle, see \cite{Guo-JFA} for details. Let $(E, \mathcal{P})$ be an $RLC$ module over $\mathbb{K}$ with base $(\Omega, \mathcal{F}, P)$, then the set $(E, \mathcal{P})^*$ of continuous module homomorphisms from $(E,\mathcal{P})$ to $L^0(\mathcal{F}, \mathbb{K})$ is the random conjugate space of $(E,\mathcal{P})$. $(E, \mathcal{P})^*$ forms an $L^0(\mathcal{F}, \mathbb{K})$-module in a natural way, denoted by $E^*$ if no confusion occurs.
\par
Let $L(E,E_1)$ be the set of continuous module homomorphisms from an $RN$ module $(E, \|\cdot\|)$ to another $RN$ module $(E_1, \|\cdot\|_1)$, where $(E,\|\cdot\|)$ and $(E_1, \|\cdot\|_1)$ are both over the scalar field $\mathbb{K}$ with base $(\Omega, \mathcal{F}, P)$, then it is well known from \cite{Guo-Modulehomo} that $T\in L(E,E_1)$ iff $T$ is an a.s. bounded linear operator, namely $T$ is linear and there exist $\xi\in L^0_+(\mathcal{F})$ such that $\|T(x)\|_1\leq \xi\|x\|$ for all $x\in E$, and $\|T\|:= \wedge \{\xi\in L^0_+(\mathcal{F})~|~\|T(x)\|_1\leq \xi\|x\|$ for all $x\in E\}$ is the $L^0$-norm of $T$. It is also known from \cite{Guo-Modulehomo} that $(L(E,E_1), \|\cdot\|)$ forms an $RN$ module over $\mathbb{K}$ with base $(\Omega, \mathcal{F}, P)$, which is still complete if $(E_1, \|\cdot\|_1)$ is complete, in particular one can see by taking $E_1=L^0(\mathcal{F},\mathbb{K})$ that $E^*$ is complete for any $RN$ module $E$. For any given $T\in L(E, E_1)$, $T^*: E^*_1 \rightarrow E^*$ defined by $T^*(f)(x)=f(T(x))$ for any $f\in E^*_1$ and $x\in E$, is the random conjugate operator of $T$, it is also known from \cite{Guo-Modulehomo} that $T^*\in L(E^*_1, E^*)$ and $\|T^*\|=\|T\|.$
\par
Random convex analysis has been developed to provide a useful analytical tool for the study of conditional convex risk measures, see \cite{Guo-JFA,Guo-SCM,GZWY,GZZ-SCM,GZZ-RCA1,GZZ-RCA2} and the references therein.
\par
In the sequel of this paper, all the $RLC$ modules occurring this paper are assumed to be over the real field $\mathbb{R}$ with base $(\Omega, \mathcal{F}, P)$.

\begin{definition}\label{definition1.4}(\cite{Guo-SCM,GZWY,GZZ-SCM,GZZ-RCA1,GZZ-RCA2})
Let $(E, \mathcal{P})$ be an $RLC$ module with base $(\Omega, \mathcal{F}, P)$. A mapping $f: E\rightarrow \bar{L}^0(\mathcal{F})$ is a proper lower semicontinuous $L^0$-convex function if the following are satisfied:
\begin{enumerate}
  \item $f(x)>-\infty$ on $\Omega$ for all $x\in E$;
  \item $dom(f):=\{x\in E~|~f(x)<+\infty \text{~on~} \Omega\}$ is nonempty;
  \item $f(\xi x+(1-\xi)y)\leq \xi f(x)+ (1-\xi)f(y)$ for any $x,y\in E$ and $\xi\in L^0_+(\mathcal{F})$ with $0\leq \xi\leq 1$ (here, we make the convention that $0\cdot(+\infty)=0$);
  \item $epi(f): = \{(x,r)\in E\times L^0(\mathcal{F})~|~f(x)\leq r\}$ is closed, where $E\times L^0(\mathcal{F})$ naturally forms an $RLC$ module with the family $\mathcal{P}'$ of $L^0$-seminorm consisting of $\{\|\cdot\|+|\cdot|: \|\cdot\|\in \mathcal{P}\}$ $((\|\cdot\|+|\cdot|)(x,r): = \|x\|+|r|$ for any $(x,r)\in E\times L^0(\mathcal{F})$ and any $\|\cdot\|\in \mathcal{P})$.
\end{enumerate}
\end{definition}
\par
Let $(E, \mathcal{P})$ be an $RLC$ module with base $(\Omega, \mathcal{F}, P)$. For any $x\in E$, define the $L^0$-seminorm $|\langle\cdot, x\rangle|: E^*\rightarrow L^0_+(\mathcal{F})$ by $|\langle\cdot, x\rangle|(f)= |f(x)|$ for any $f\in E^*$, then $\sigma(E^*,E): = \{ |\langle\cdot, x\rangle|: x\in E\}$ is a family of $L^0$-seminorms such that $(E^*, \sigma(E^*,E))$ is an $RLC$ module with base $(\Omega, \mathcal{F}, P)$, whose $(\varepsilon,\lambda)$-topology is the random $w^*$-topology for $E^*$, a proper lower semicontinuous $L^0$-convex function from $(E^*, \sigma(E^*,E))$ to $\bar{L}^0(\mathcal{F})$ is a proper random $w^*$-lower semicontinuous $L^0$-convex function on $E^*$. If $f: E \rightarrow \bar{L}^0(\mathcal{F})$ is a proper lower semicontinuous $L^0$-convex function, as usual, we can define the random (Fenchel) conjugate $f^*$ of $f$ by $f^*(u)= \vee \{u(x)-f(x): x\in E\}$ for any $u\in E^*$, then $f^*$ is just a proper random $w^*$-lower semicontinuous $L^0$-convex function on $E^*$. Similarly, $f_E^{**}: E\rightarrow \bar{L}^0(\mathcal{F})$ is defined by $f_E^{**}(x)= \vee \{u(x)-f^*(u)~|~u\in E^*\}$ for any $x\in E$, then one can have $f_E^{**}=f$, namely the random Fenchel-Moreau duality theorem, which was established in \cite{GZZ-RCA1}. When $E$ is an $RN$ module with base $(\Omega, \mathcal{F}, P)$, we denote the random conjugate space of $E^*$ by $E^{**}$, and further define $f^{**}: E^{**} \rightarrow \bar{L}^0(\mathcal{F})$ by $f^{**}(g)=\vee \{g(u)- f^*(u): u\in E^*\}$ for any $g\in E^{**}$, then it is clear that $f_E^{**}=f^{**}|_E$ (the restriction of $f^{**}$ to $E$).

In the sequel of this section, $(E, \|\cdot\|)$ always denotes a complete $RN$ module with base $(\Omega, \mathcal{F}, P)$ such that $(E, \|\cdot\|)$ has full support (see the related discussion in Section 2), $\mathscr{C}(E)$ the set of proper lower semicontinuous $L^0$-convex functions on $E$, and $\mathscr{C}_{w^*}(E^*)$ the set of proper random $w^*$-lower semicontinuous $L^0$-convex functions on $E^*$. We observe that $\mathscr{C}(E)$ is stable in the sense that $\tilde{I}_A f_1+\tilde{I}_{A^c} f_2\in \mathscr{C}(E)$ for any $A\in \mathcal{F}, f_1$ and $f_2$ in $\mathscr{C}(E)$, where $A^c=\Omega\setminus A$, $I_A$ stands for the characteristic function of $A$ and $\tilde{I}_A$ the equivalence class of $I_A$, and $\tilde{I}_A f_1+\tilde{I}_{A^c} f_2: E\rightarrow \bar{L}^0(\mathcal{F})$ is defined by $[\tilde{I}_A f_1+\tilde{I}_{A^c} f_2](x)=\tilde{I}_A f_1(x)+\tilde{I}_{A^c} f_2(x), \forall x\in E$. Similarly, one can understand that $\mathscr{C}_{w^*}(E^*)$ is stable.
\par
We can now define the order preserving (reversing) operators acting on $\mathscr{C}(E)$ and then state the main results of this paper. As usual, for any $f_1$ and $f_2$ in $\mathscr{C}(E)$ ($g_1$ and $g_2$ in $\mathscr{C}_{w^*}(E^*)$), $f_1\leq f_2$ means $f_1(x)\leq f_2(x)$ for any $x\in E$ ($g_1\leq g_2$ means $g_1(u)\leq g_2(u)$ for any $u\in E^*$).

\begin{definition}\label{definition1.5}
An operator $T: \mathscr{C}(E) \rightarrow \mathscr{C}(E)$ is
\begin{enumerate}
  \item order preserving if $f\leq g$ implies $T(f)\leq T(g)$.
  \item fully order preserving if $f\leq g$ iff $T(f)\leq T(g)$, and additionally $T$ is onto.
  \item stable if $T(\tilde{I}_Af+ \tilde{I}_{A^c}g)= \tilde{I}_A T(f)+ \tilde{I}_{A^c} T(g)$ for any $A\in \mathcal{F}$ and any $f$ and $g$ in $\mathscr{C}(E)$.
  \item an involution if $T(T(f))=f$ for any $f\in \mathscr{C}(E)$.
\end{enumerate}
\end{definition}

\par
Our first main result is the following:

\begin{theorem}\label{OP}
An operator $T: \mathscr{C}(E) \rightarrow \mathscr{C}(E)$ is stable and fully order preserving iff there exist $c\in E, w\in E^*, \beta\in L^0(\mathcal{F}), \tau\in L^0_{++}(\mathcal{F})$ and a continuous module automorphism $H$ of $E$ such that $T(f)(x)=\tau f(H(x)+c)+w(x)+ \beta$ for any $x\in E$ and $f\in \mathscr{C}(E)$.
\end{theorem}

\begin{corollary}\label{involution}
An operator $T: \mathscr{C}(E) \rightarrow \mathscr{C}(E)$ is a stable order preserving involution iff there exist $c\in E, w\in E^*$ and a continuous module automorphism $H$ of $E$, satisfying $H^2= ID_X, c\in Ker(H+ ID_E), w\in Ker(H^*+ ID_{E^*})$ such that
$$T(f)(x)=f(H(x)+c)+ w(x)-\frac{1}{2}w(c)$$
for all $x\in E$, where $ID_E$ and $ID_{E^*}$ stand for the identity operators on $E$ and $E^*$, respectively.
\end{corollary}

\begin{definition}\label{definition1.8}
An operator $S: \mathscr{C}(E) \rightarrow \mathscr{C}_{w^*}(E^*)$ is
\begin{enumerate}
  \item order reversing if $f\leq g$ implies $S(f)\geq S(g)$.
  \item fully order reversing whenever $f\leq g$ iff $S(f)\geq S(g)$, and additionally $S$ is onto.
  \item stable if $S(\tilde{I}_A f+ \tilde{I}_{A^c} g)= \tilde{I}_A S(f)+ \tilde{I}_{A^c} S(g)$ for any $A\in \mathcal{F}$ and $f$ and $g$ in $\mathscr{C}(E)$.
\end{enumerate}
\end{definition}

\par
Our second main result is the following:

\begin{theorem}\label{OR}
An operator $S: \mathscr{C}(E) \rightarrow \mathscr{C}_{w^*}(E^*)$ is stable fully order reversing iff there exist $v\in E^*, y\in E, \rho\in L^0(\mathcal{F}), \tau \in L^0_{++}(\mathcal{F})$ and a continuous module automorphism $H$ of $E$ such that $S(f)(u)= \tau f^*(H^*(u)+ v)+ u(y)+ \rho$ for all $f\in \mathscr{C}(E)$ and all $u\in E^*.$
\end{theorem}

The main results of this paper, whether from their formulations or from the ideas of their proofs, are considerably motivated by the work in \cite{AM,IRS}, but we would like to mention here the main differences between the work in \cite{AM,IRS} and our work.

The foundation for the work in \cite{AM,IRS} is different from that for our work. The work in \cite{AM,IRS} depends on classical convex analysis, which is based on the theory of conjugate spaces for locally convex spaces, and the fundamental theorem of affine geometry in linear spaces, whereas our work depends on random convex analysis, which is based on the theory of random conjugate spaces for random locally convex modules, and has just been developed for the last ten years \cite{GZWY,GZZ-RCA1,GZZ-RCA2}, and the fundamental theorem of affine geometry in regular $L^0$-modules, which was just established in \cite{WGL}.

We mention now the second difference between our approach and the one in \cite{IRS}. Let $\mathscr{A}(E)$ be the $L^0$-module of continuous $L^0$-affine functions on a complete $RN$ module $E$ with base $(\Omega, \mathcal{F}, P)$, as in \cite{IRS} for a fully order preserving operator $T$ from $\mathscr{C}(X)$ to $\mathscr{C}(X)$ when $X$ is a Banach space, we will first prove that for a stable fully order preserving operator $T$ from $\mathscr{C}(E)$ to $\mathscr{C}(E)$ in the setting of Theorem \ref{OP}, $T|_{\mathscr{A}(E)}$ is a stable bijection from $\mathscr{A}(E)$ onto $\mathscr{A}(E)$, and hence induces a stable bijection $\hat{T}$ from $E^*\times L^0(\mathcal{F})$ onto $E^*\times L^0(\mathcal{F})$ by identifying $\mathscr{A}(E)$ with $E^*\times L^0(\mathcal{F})$, then we will directly prove that $\hat{T}$ is $L^0$-affine as a whole by applying the fundamental theorem of affine geometry in regular $L^0$-modules \cite{WGL}, not so in coordinate component manner as in \cite{IRS}. Our approach has an advantage that many complicated discussions on the rank of $L^0$-modules can be avoided.

The third difference is that the operators in Theorem \ref{OP}, Corollary \ref{involution} and Theorem \ref{OR} are additionally assumed to be stable, the additional assumption is natural and necessary from the process of establishing the fundamental theorem of affine geometry in regular $L^0$-modules \cite{WGL}. When $(\Omega, \mathcal{F}, P)$ is trivial, namely $\mathcal{F}=\{\Omega, \emptyset\}$ a complete $RN$ module $E$ with base $(\Omega, \mathcal{F}, P)$ just reduces to a Banach space, in which case an order preserving (or reversing) operator (in fact any operator) is automatically stable, that is to say, the stable assumption automatically disappears. On the other hand, when $(\Omega, \mathcal{F}, P)$ is a general probability space, $(L^0(\mathcal{F}), \leq)$ is a partially ordered set rather than a totally ordered set like $\mathbb{R}$, which also makes the proofs of the main results of this paper more involved than those of \cite{IRS}.
\par
The remainder of this paper is organized as follows. In Section 2 of this paper, we give some preliminaries mainly on $L^0$-affine functions. In Section 3 of this paper, we treat the stable order preserving operators and prove Theorem \ref{OP} and Corollary \ref{involution}. Finally, in Section 4 of this paper we treat the stable order reserving operator and prove Theorem \ref{OR}.

\section{Preliminaries}

 For each $A\in {\mathcal F}$, $\tilde A=\{B\in {\mathcal F}: P(B\triangle A)=0\}$ denotes the equivalence class of $A$, and we also use $I_{\tilde A}$ to stand for $\tilde I_A$. Given $\xi\in L^0({\mathcal F})$, the notation $[\xi>0]$ stands for the equivalence class of $A$, where $A=\{\omega\in \Omega~|~\xi^0(\omega)>0\}$ and $\xi^0$ is an arbitrarily chosen representative of $\xi$, thus $I_{[\xi>0]}=\tilde{I}_A$. Some other notation like $[\xi= 0]$ and so on are understood in the similar way.

Besides, we denote $\mathcal F_+=\{A\in {\mathcal F}: P(A)>0\}$.

In our main results, Theorem \ref{OP}, Corollary \ref{involution} and Theorem \ref{OR}, the $RN$ module is required to have full support. Now we recall the notion of support for an $RN$ module.
Let $(E,\|\cdot\|)$ be an $RN$ module with base $(\Omega,{\mathcal F},P)$. Denote $\xi=\vee\{\|x\|: x\in E\}$(generally $\xi\in \bar L^0({\mathcal F})$). Clearly $[\xi>0]=[\xi=+\infty]$. Any representative $A$ of $[\xi>0]$ is called a support of $(E,\|\cdot\|)$. If $\Omega$ is a support of $(E,\|\cdot\|)$, then $(E,\|\cdot\|)$ is called having full support. Except for the trivial case $E=\{\theta\}$, a support $A$ of $(E,\|\cdot\|)$ must have positive probability, if $(E,\|\cdot\|)$ does not have full support, then $A^c$ has positive probability and $\tilde I_{A^c}x=\theta$ for every $x\in E$, therefore $A^c$ is indeed redundant for $(E,\|\cdot\|)$, in which case we define a probability measure $P_A: {\mathcal F}\cap A\to [0,1]$ by $P_A(B\cap A)=\frac{P(B\cap A)}{P(A)}, \forall B\in {\mathcal F}$, then since $\|x\|=\tilde I_A\|x\|$ for every $x\in E$, we can always take $(E,\|\cdot\|)$ as an $RN$ module with base the smaller probability space $(A,{\mathcal F}\cap A, P_A)$, in which case $(E,\|\cdot\|)$ has full support. We see that the requirement of full support is not an excessive restriction.

In the sequel of this paper, an element $x$ of an $RN$ module $(E,\|\cdot\|)$ is said to have full support, if $\|x\|\neq 0$ on $\Omega$, in other words, $\|x\|\in L^0_{++}({\mathcal F})$. For any $x\in E$ and $u\in E^*$, we write $\langle u,x\rangle$ for $u(x)$. As usual, $E^{**}=(E^*)^*$, and we can regard $E$ as a subset of $E^{**}$ by the natural embedding.

\begin{proposition}\label{support}
Let $(E,\|\cdot\|)$ be an $RN$ module and $E^*$ its random conjugate space. Suppose that $(E,\|\cdot\|)$ is complete and has full support. Then the following statements are true:\\
(1). There exists $x_0\in E$ such that $\|x_0\|=1$, in addition, there exists $u_0\in E^*$ such that $\langle u_0,x_0\rangle=1 $ and $\|u_0\|=1$.\\
(2). If $u\in E^*$ has full support, then for any $\xi\in L^0(\mathcal F)$, there exists $x_0\in E$ such that $\langle u,x_0\rangle=\xi$.
\end{proposition}

\begin{proof}
(1). The first part is shown in \cite[Lemma 3.1]{Guo-BBKS}. The second part is \cite[Corollary 3.2]{Guo-survey1}, see also \cite[Remark 2.2]{Guo-BBKS}.\\
(2). Consider the family $\{|\langle u,x\rangle|: x\in E, \|x\|\leq 1\}$, using the same technique as in the proof of \cite[Lemma 3.1]{Guo-BBKS}, there exists $x_u\in E$ such that $\langle u,x_u\rangle=\vee\{|\langle u,x\rangle|: x\in E, \|x\|\leq 1\}=\|u\|$. If $u$ has full support, then $\|u\|$ is invertible in $L^0({\mathcal F})$, as a result $x_0=\xi \|u\|^{-1} x_u$ satisfies $\langle u,x_0\rangle=\xi$.
\end{proof}

Now we discuss the indicator function.

For any $x_0\in E$, we define the indicator function $\delta_{x_0}: E\to \bar L^0({\mathcal F})$ by $$\delta_{x_0}(x)=0\cdot I_{[\|x-x_0\|=0]}+\infty\cdot I_{[\|x- x_0\|\neq 0]}, \forall x\in E.$$
 Note that $\delta_{x_0}(x_0)=0$, however if $x\neq x_0$, $\delta_{x_0}(x)$ may be not equal to $+\infty$, since it is possible that $\|x-x_0\|=0$ on some $A\in {\mathcal F}_+$ so that $\delta_{x_0}(x)=0$ on $A$.

 By simple analysis and calculation, we can obtain the following:

\begin{lemma}\label{Indicator}
 For any $x_0\in E$, the indicator function $\delta_{x_0}$ is a member of ${\mathscr C}(E)$ with $dom(\delta_{x_0})=\{x_0\}$. The random conjugate $\delta^*_{x_0}: E^*\to \bar L^0({\mathcal F})$ is given by $\delta^*_{x_0}(u)=\langle u, x_0\rangle, \forall u\in E^*$, and the bi-random conjugate $\delta^{**}_{x_0}:E^{**}\to \bar L^0({\mathcal F})$ is given by $\delta^{**}_{x_0}(x^{**})=0\cdot I_{[\|x^{**}-x_0\|=0]}+\infty\cdot I_{[\|x^{**}- x_0\|\neq 0]}=\delta_{x_0}(x^{**}), \forall x^{**}\in E^{**}.$
 \end{lemma}

 Now, we turn to $L^0$-affine functions.

A function $h: E\to L^0({\mathcal F})$ is called an $L^0$-affine function, if there exists a pair $(u,\alpha)\in E^*\times L^0({\mathcal F})$ such that $h(x)=\langle u,x\rangle+\alpha, \forall x\in E$. From now on, such an $L^0$-affine function $h$ will be written as $h_{u,\alpha}$. Denote by ${\mathscr A}(E)$ the set of $L^0$-affine functions on $E$. Obviously, ${\mathscr A}(E)$ is a subset of ${\mathscr C}(E)$.

Given a family of functions $f_i: E\to \bar L^0({\mathcal F}), i\in I$, define $\vee_{i\in I} f_i: E\to \bar L^0({\mathcal F})$ by $[\vee_{i\in I} f_i](x)=\vee\{f_i(x):i\in I\}, \forall x\in E$. Specially we write $f\vee g$ for $\vee\{f,g\}$.

For $f\in {\mathscr C}(E)$, define ${\mathscr A}(f)\subset {\mathscr A}(E)$ as ${\mathscr A}(f)=\{h\in {\mathscr A}(E): h\leq f\}$.

We will present a series of useful properties for $L^0$-affine functions in the sequel of this section.

Corollary \ref{Affsup} below is indeed obtained in the proof of random Fenchel-Moreau duality theorem \cite[Theorem 5.1]{GZZ-RCA1}(see also \cite{GZWY}). Now we deduce it from the random Fenchel-Moreau duality theorem.

 Fix $f\in {\mathscr C}(E)$. For any $u\in E^*$, it is easy to verify that $$\vee\{\langle u,x\rangle-f(x):x\in E\}=\vee\{\langle u,x\rangle-f(x): x\in dom(f)\},$$
 therefore $$f^*(u)=\vee\{\langle u,x\rangle-f(x): x\in dom(f)\},\forall u\in E^*.$$ Similarly, we have $$f^{**}_E(x)=\vee\{\langle u,x\rangle-f^*(u): u\in dom(f^*)\},\forall x\in E.$$
 By the random Fenchel-Moreau duality theorem, namely $f_E^{**}=f$, we can see the following:

\begin{corollary}\label{Affsup}
 For any $f\in {\mathscr C}(E)$, ${\mathscr A}(f)$ is nonempty and $f=\vee\{h: h\in{\mathscr A}(f)\}$.
\end{corollary}

Using $L^0$-affine functions, we can give a useful characterization for the bi-random conjugate $f^{**}$, which is an analogue of \cite[Proposition 10 (ii)]{IRS}.

\begin{proposition}\label{bi-conjugate}
For any $f\in {\mathscr C}(E)$ and every $a\in E^{**}$, $f^{**}(a)=\vee\{\langle a, u\rangle+\alpha: h_{u,\alpha}\in {\mathscr A}(f)\}$.
\end{proposition}
\begin{proof}
By definition, $f^{**}(a)=\vee\{\langle a, u\rangle-f^*(u): u\in E^*\}=\vee\{\langle a, u\rangle-f^*(u): u\in dom (f^*)\}$, also $f^*(u)=\vee\{\langle u, x\rangle-f(x):x\in E\}$, so that for each $u\in dom (f^*)$, $h_{u,-f^*(u)}(x)=\langle u, x\rangle-f^*(u)$ for all $x\in E$, and hence $h_{u,-f^*(u)}$ belongs to ${\mathscr A}(f)$ for all $u\in dom (f^*)$. Next, we prove that if $h_{u,\alpha}$ belongs to
${\mathscr A}(f)$ then $u\in dom (f^*)$ and $\alpha\leq -f^*(u)$. Indeed, if $h_{u,\alpha}\in {\mathscr A}(f)$, then for all $x\in E$, $\langle u, x\rangle+\alpha\leq f(x)$, therefore $\alpha\leq \wedge\{f(x)-\langle u, x\rangle: x\in E\}\in L^0({\mathcal F})$, then $f^*(u)=\vee\{\langle u, x\rangle-f(x):x\in E\}=-\wedge\{f(x)-\langle u, x\rangle: x\in E\}\in L^0({\mathcal F})$ and $-f^*(u)\geq \alpha$. We conclude that
$\vee\{\langle a, u\rangle+\alpha: h_{u,\alpha}\in {\mathscr A}(f)\}=\vee\{\langle a, u\rangle-f^*(u): u\in dom (f^*)\}=f^{**}(a)$.
\end{proof}

We present a comparison theorem for two $L^0$-affine functions as follows.

\begin{proposition}\label{comparison}
 Suppose that $h_{u,\alpha}$ and $h_{v,\beta}$ are two $L^0$-affine functions on $E$ such that $h_{v,\beta}\leq h_{u,\alpha}$, then $v=u$ and $\beta\leq\alpha$.
\end{proposition}

\begin{proof}
According to definition, $h_{v,\beta}\leq h_{u,\alpha}$ means that $\langle v,x\rangle+\beta\leq \langle u,x\rangle+\alpha, \forall x\in E$. Taking $x=\theta$ yields that $\beta\leq\alpha$. For any $n\in \mathbb N$, replacing $x$ with $nx$, we can deduce that $\langle v,x\rangle+\frac{\beta}{n}\leq \langle u,x\rangle+\frac{\alpha}{n}, \forall x\in E$. Letting $n$ tend to $\infty$, we obtain that $\langle v,x\rangle\leq \langle u,x\rangle, \forall x\in E$. Then replacing $x$ with $-x$ we get $-\langle v,x\rangle\leq -\langle u,x\rangle, \forall x\in E$. Therefore $\langle v,x\rangle= \langle u,x\rangle, \forall x\in E$, namely $v=u$.
\end{proof}

A useful characterization for $f\in {\mathscr C}(E)$ to belong to ${\mathscr A}(E)$ is given as follows.

\begin{proposition}\label{AffCha}
 Given $f\in {\mathscr C}(E)$, then $f$ belongs to ${\mathscr A}(E)$ iff there exists a unique pair $(u,\alpha)\in E^*\times L^0({\mathcal F})$ such that ${\mathscr A}(f)=\{h_{u,\beta}: \beta\leq \alpha\}$.
\end{proposition}

\begin{proof}
 (1). Necessity. If $f\in {\mathscr A}(E)$, then there exists a unique pair $(u,\alpha)\in E^*\times L^0({\mathcal F})$ such that $f=h_{u,\alpha}$. If $h_{v,\beta}\in {\mathscr A}(E)$ satisfies that $h_{v,\beta}\leq h_{u,\alpha}$, then it follows from Proposition \ref{comparison} that $v=u$ and $\beta\leq \alpha$, thus ${\mathscr A}(f)=\{h_{u,\beta}: \beta\leq \alpha\}$.\\
(2). Sufficiency. According to Corollary \ref{Affsup},  ${\mathscr A}(f)$ is nonempty and $f=\vee\{h: h\in {\mathscr A}(f)\}$. If there exists a unique pair $(u,\alpha)\in E^*\times L^0({\mathcal F})$ such that ${\mathscr A}(f)=\{h_{u,\beta}: \beta\leq \alpha\}$, then $f =\vee\{h_{u,\beta}: \beta\leq \alpha\}=h_{u,\alpha}\in {\mathscr A}(E)$.
\end{proof}

\begin{corollary}
Assume that $h_{u,\alpha}\in {\mathscr A}(E)$ and $f\in {\mathscr C}(E)$. If $ f \leq ¡Üh_{u,\alpha}$, then $f=h_{u,\beta}$ for some $\beta\leq \alpha$.
\end{corollary}

\begin{proof}
   If $h_{u^\prime,\alpha^\prime}\in {\mathscr A}(f)$, then $h_{u^\prime,\alpha^\prime}\leq f\leq  h_{u,\alpha}$, i.e. $h_{u^\prime,\alpha^\prime}\in {\mathscr A}(h_{u,\alpha})$. By Proposition \ref{comparison}, $u^\prime=u$ and $\alpha^\prime \leq \alpha$. Let $\beta=\vee\{\alpha^\prime: h_{u^\prime,\alpha^\prime}\in {\mathscr A}(f)\}$, then $\beta\leq \alpha$, thus $\beta\in L^0({\mathcal F})$. It is easy to see that  $f=h_{u,\beta}$.
\end{proof}

Proposition \ref{subdiff} below is substantial in {\it Step 2} of the proof of the following Proposition \ref{line-line}, which is the foundation of our proof of Theorem \ref{OP}. To prove Proposition \ref{subdiff}, we are forced to use a separation theorem that relies on the random $w^*$-closedness of an $L^0$-convex subset of $E^*$. Lemma \ref{Wsclosed} below is an auxiliary result which says that the $L^0$-convex hull of any finite subset of an $RLC$ module is closed. Since $L^0$-convex sets plays a crucial role in random convex analysis on $RLC$ modules, Lemma \ref{Wsclosed} is a by-product which may be interesting on its own.

\begin{lemma}\label{Wsclosed}
 Let $(E, \mathcal{P})$ be an $RLC$ module with base $(\Omega,{\mathcal F},P)$. Then for any finite subset $\{x_1,x_2, \dots, x_d\}$ of $E$, its $L^0$-convex hull $Conv_{L^0}\{x_1,x_2, \dots, x_d\}$
 $$:= \{\lambda_1 x_1+\lambda_2x_2+\cdots+\lambda_dx_d: \lambda_i\in L^0_+({\mathcal F}),i=1,2,\dots,d, \sum^d_{i=1}\lambda_i=1 \} $$
is a complete thus closed subset of $E$.
\end{lemma}
\begin{proof}
 Consider the submodule $M$ of $E$ generated by $\{x_1,x_2,\dots,x_d\}$, namely,
 $$M=span_{L^0}\{x_1,x_2, \dots, x_d\}:=\{\lambda_1 x_1+\lambda_2x_2+\cdots+\lambda_dx_d: \lambda_i\in L^0_+({\mathcal F}), i=1,2,\dots,d\},$$
 inherited the topology from $E$.

 By \cite[Proposition 2.3]{WGL}, there exists a finite partition $\{A_0,A_1,\dots,A_d\}$ of $\Omega$ to ${\cal F}$ such that ${\tilde I}_{A_i}M$ is a free module of rank $i$ over the algebra ${\tilde I}_{A_i}L^0(\mathcal F)$ for each $i\in\{0,1,\dots,d\}$ satisfying $P(A_i)>0$, in which case $M=\bigoplus^d_{i=0}{\tilde I}_{A_i}M$. We can assume, without loss of generality, that $P(A_i)>0$ for all $i\in \{1,2\dots,d\}$. According to (1) of \cite[Lemma 3.4]{GP}, for each $i\in \{1,2\dots,d\}$, ${\tilde I}_{A_i}M$ and ${\tilde I}_{A_i}L^0(\mathcal F, \mathbb R^i)$ are isomorphic in the sense of topological modules, and hence also a complete subset of $E$.

 For each $i\in \{1,2\dots,d\}$, suppose that $T_i$ is the topological module isomorphism from ${\tilde I}_{A_i}M$ to ${\tilde I}_{A_i}L^0(\mathcal F, \mathbb R^i)$. By regarding $L^0(\mathcal F, \mathbb R^i)$ as the submodule $L^0(\mathcal F, \mathbb R^i)\times \{0\}^{d-i}$ of $L^0(\mathcal F, \mathbb R^d)$ for each $i\in \{1,\dots,d-1\}$, the mapping $T: M\to L^0(\mathcal F, \mathbb R^d)$ given by $$Tx=T(\sum^n_{i=0}{\tilde I}_{A_i}x)=\sum^d_{i=1}{\tilde I}_{A_i}T_i({\tilde I}_{A_i}x),\forall x\in M$$
 is a topological module homomorphism such that $M$ and $TM$ is topological module isomorphic. From now on, we identify $M$ with $TM$ by regarding $x_1,x_2,\dots,x_d$ as elements of $L^0(\mathcal F, \mathbb R^d)$ for the simplicity of notation! It suffices to show that $Conv_{L^0}\{x_1,x_2, \dots, x_d\}$ is a complete subset of $L^0(\mathcal F, \mathbb R^d)$.

 For any Cauchy sequence $\{y_n, n\in\mathbb N\}$ in $Conv_{L^0}\{x_1,x_2, \dots, x_d\}$, where each $y_n$ can be expressed as $y_n=\lambda^1_nx_1+\cdots+\lambda^d_nx_d$ for some $\lambda^i_n\in L^0_+({\mathcal F}), i=1,2,\dots,d, \sum^d_{i=1}\lambda^i_n=1$. Due to the completeness of $L^0(\mathcal F, \mathbb R^d)$, $y_n$ converges in probability to some $y\in L^0(\mathcal F, \mathbb R^d)$. It remains to show $y\in Conv_{L^0}\{x_1,x_2, \dots, x_d\}$. By passing to a suitable subsequence, we may assume without loss of generality that $y_n$  converges $P$-a.s. to $y$.
 Since $\{\lambda_n=(\lambda^1_n,\dots,\lambda^d_n), n\in \mathbb N\}$ is a.s.bounded in $L^0(\mathcal F, \mathbb R^d)$, by a variant of Komlos' principle of subsequences \cite[Lemma 1.61]{FS}, there exists a sequence of convex combinations $$\mu_n\in conv\{\lambda_{n},\lambda_{n+1},\dots\}$$ which converges $P$-a.s. to some $\mu\in L^0(\mathcal F, \mathbb R^d)$. For each $n$, assume that $\mu^i_n=t^n_0\lambda^i_n+\cdots+t^n_{l(n)}\lambda^i_{n+l(n)}, i=1,2\dots,d$, where $l(n)$ is some positive integer and $t^n_0,\cdots,t^n_{l(n)}$ are some non-negative numbers such that $\sum\limits_{k=0}^{l(n)}t^n_k=1$. Clearly $\mu^i_n\in L^0_+({\mathcal F}), i=1,\dots,d$ and $\mu^1_n+\cdots+\mu^d_n=1$, combining with the fact that $\mu_n=(\mu^1_n,\dots,\mu^d_n)$ converges $P$-a.s to $\mu=(\mu^1,\dots,\mu^d)$, we deduce that $\mu^i\in L^0_+({\mathcal F}), i=1,\dots,d$ and $\mu^1+\cdots+\mu^d=1$. Since $y_n$ converges $P$-a.s. to $y$, the forward convex combination sequence $z_n:=\sum\limits_{k=0}^{l(n)}t^n_ky_{n+k}$ also converges $P$-a.s. to $y$. By simple calculation, we see that $z_n=\mu^1_nx_1+\cdots+\mu^d_nx_d$. Then the fact that $\mu_n$ converges $P$-a.s. to $\mu$ implies that $z_n$ converges $P$-a.s. to $\mu^1 x_1+\cdots+\mu^dx_d$. We conclude that $y=\mu^1 x_1+\cdots+\mu^dx_d$ belongs to $Conv_{L^0}\{x_1,x_2, \dots, x_d\}$.
\end{proof}

\begin{remark}
We point out that the proof of Lemma \ref{Wsclosed} is similar to that of \cite[Theorem 3.5]{Wu}, which states that for any finite subset $\{x_1,x_2, \dots, x_d\}$ of an RLC module $E$, the finitely generated $L^0$-convex cone $$cone_{L^0}\{x_1,\dots,x_d\}:=\{\lambda_1 x_1+\cdots+\lambda_d x_d: \lambda_i\in L^0_+({\mathcal F}),i=1,\dots,d\}$$
is a complete and thus closed subset of $E$.
\end{remark}

We now state and prove Proposition \ref{subdiff}. We remind the reader that for any $\xi$ and $\eta$ in $L^0({\mathcal F})$, $\xi>\eta$ means $\xi\geq \eta$ and $\xi\neq \eta$.

\begin{proposition}\label{subdiff}
Let $(E,\|\cdot\|)$ be an $RN$ module with base $(\Omega,{\mathcal F},P)$, and $E^*$ its random conjugate space. Suppose that $h_{u,\alpha}, h_{u_1,\alpha_1}$ and $h_{u_2,\alpha_2}$ are $L^0$-affine functions on $E$ such that $h_{u,\alpha}\leq h_{u_1,\alpha_1}\vee h_{u_2,\alpha_2}$, then $\langle u,x\rangle\leq (\langle u_1,x\rangle \vee \langle u_2,x\rangle), \forall x\in E$, and there exists $\mu\in L^0({\mathcal F})$ with $0\leq \mu\leq 1$ such that $u=\mu u_1+(1-\mu) u_2$.
\end{proposition}

\begin{proof}
By definition, $h_{u,\alpha}\leq h_{u_1,\alpha_1}\vee h_{u_2,\alpha_2}$ means that $\langle u,x\rangle+\alpha\leq (\langle u_1,x\rangle +\alpha_1) \vee (\langle u_2,x\rangle+\alpha_2), \forall x\in E$.

For any $n\in \mathbb N$, replacing $x$ with $nx$, we can deducing that $\langle u,x\rangle+\frac{\alpha}{n}\leq (\langle u_1,x\rangle +\frac{\alpha_1}{n}) \vee (\langle u_2,x\rangle+\frac{\alpha_2}{n}), \forall x\in E$. Letting $n$ tend to $\infty$, we obtain that
\begin{equation}\label{P10:1}
\langle u,x\rangle\leq (\langle u_1,x\rangle \vee \langle u_2,x\rangle), \forall x\in E
\end{equation}

Consider the $RLC$ module $(E^*, \sigma(E^*,E))$, it follows from Lemma \ref{Wsclosed} that $conv_{L^0}\{u_1, u_2\}$ is an $L^0$-convex and closed subset of $(E^*, \sigma(E^*,E))$. We prove $u\in conv_{L^0}\{u_1, u_2\}$ by contradiction. If $u\notin conv_{L^0}\{u_1, u_2\}$, then by the separation theorem in $RLC$ modules \cite[Theorem 3.1]{GXC} (see also \cite[Theorem 3.6]{Guo-JFA}), there exists $g\in (E^*, \sigma(E^*,E))^*$ such that
\begin{equation}\label{P10:2}
\langle u,g\rangle> \vee\{\langle w,g\rangle: w\in conv_{L^0}\{u_1, u_2\}\}.
\end{equation}
\eqref{P10:2} also means there exists some $B\in \mathcal{F}$ with $P(B)>0$ such that
 \begin{equation}\label{P10:3}
\langle u,g\rangle> \vee\{\langle w,g\rangle: w\in conv_{L^0}\{u_1, u_2\}\} \text{~on~$B$}.
\end{equation}

Since $(E^*, \sigma(E^*,E))$ is endowed with the $(\varepsilon,\lambda)$-topology in this paper, it is well known that the random conjugate space of $(E^*, \sigma(E^*,E))$ under the locally $L^0$-convex topology is $E$ (see \cite{Guo-JFA,GZZ-RCA1} for the notion of the locally $L^0$-convex topology and the related random conjugate space), then by \cite[Proposition 2.6]{GZZ-RCA1}, the random conjugate space under the $(\varepsilon,\lambda)$-topology of $(E^*, \sigma(E^*,E))$ is $H_{cc}(E)$, where $H_{cc}(E)$ stands for the countable concatenation hull of $E$ in the random conjugate space under the $(\varepsilon,\lambda)$-topology of $(E^*, \sigma(E^*,E))$. Thus there exist a sequence $\{x_n, n\in \mathbb{N}\}$ in $E$ and a countable partition $\{A_n, n\in \mathbb{N}\}$ of $\Omega$ to $\mathcal{F}$ such that $\langle u,g\rangle =\sum_{n=1}^{\infty} \tilde{I}_{A_n} \langle u, x_n\rangle$, which must imply by (3) that there exists some $n_0\in \mathbb{N}$ such that $P(A_{n_0}\cap B)> 0$ and $\langle u, x_{n_0}\rangle > \vee \{\langle w, x_{n_0}\rangle: w\in Conv_{L^0} \{u_1,u_2\}\}$ on $A_{n_0}\cap B$. Specially $\langle u,x_{n_0}\rangle> (\langle u_1,x_{n_0}\rangle)\vee (\langle u_2, x_{n_0}\rangle)$ on $A_{n_0} \cap B$, which is a contradiction to \eqref{P10:1}. We conclude that $u\in conv_{L^0}\{u_1, u_2\}$, namely there exists $\mu\in L^0({\mathcal F})$ with $0\leq \mu\leq 1$ such that $u=\mu u_1+(1-\mu) u_2$.
\end{proof}

\section{The order preserving case}\label{OPCase}

In this section, we prove Theorem \ref{OP} and Corollary \ref{involution}. For this purpose, we need some preparations. We will gradually establish a series of characteristic properties for a stable and fully order preserving operator.

\begin{proposition}\label{inverse}
  If $T: \mathscr{C}(E) \rightarrow \mathscr{C}(E)$ is a stable and fully order preserving operator, then so is its inverse $T^{-1}$.
\end{proposition}

\begin{proof}
It is obvious.
\end{proof}

\begin{proposition}\label{exchange}
  If $T: \mathscr{C}(E) \rightarrow \mathscr{C}(E)$ is stable and fully order preserving, then for any family $\{f_i\}_{i\in I}\subset {\mathscr C}(E)$ such that $\vee_{i\in I} f_i$ belongs to ${\mathscr C}(E)$, we have that $T(\vee_{i\in I} f_i) =\vee_{i\in I} T(f_i)$.
\end{proposition}

The proof is omitted, since it is a word-by-word copy of the proof of \cite [Proposition 1 (ii)]{IRS}. Note that in the proof, the assumption that $T$ is stable is not used.

We then show that each stable and fully order preserving operator maps any $L^0({\mathcal F})$-valued function to an $L^0({\mathcal F})$-valued function. Precisely, we have the following:

\begin{proposition}\label{valuefinite}
 If $T: \mathscr{C}(E) \rightarrow \mathscr{C}(E)$ is stable and fully order preserving and $f\in {\mathscr C}(E)$ with $f(x)\in L^0({\mathcal F})$ for every $x\in E$, then $[T(f)](x)\in L^0({\mathcal F})$ for every $x\in E$.
\end{proposition}
\begin{proof}
We prove this proposition by contradiction. Suppose that there exist $x_0\in E$ and $A\in {\mathcal F}_+$ such that $[T(f)](x_0)=\infty$ on $A$. According to Lemma \ref{Indicator}, the indicator function $\delta_{x_0}$ is a member of ${\mathscr C}(E)$. Since $T$ is onto, there exists $f_0\in {\mathscr C}(E)$ such that $T(f_0)=\delta_{x_0}$. Due to the assumption on $f$ we see that $f_1=f\vee f_0$ belongs to ${\mathscr C}(E)$. Then using Proposition \ref{exchange} we get $T(f_1)=T(f)\vee T(f_0)=T(f)\vee \delta_{x_0}$. If $\tilde x\in dom(T(f_1))$, then it follows from $\delta_{x_0}(\tilde x)\leq [T(f_1)](\tilde x)\in L^0({\mathcal F})$ that $\tilde x=x_0$, subsequently, $[T(f_1)](\tilde x)\geq [T(f)](\tilde x)=[T(f)](x_0)=\infty$ on $A$, this implies that $dom(T(f_1))$ must be empty and thus $T(f_1)\notin {\mathscr C}(E)$, which is a contradiction.
\end{proof}

Since we have known from Corollary \ref{Affsup} that every $f\in {\mathscr C}(E)$ can be expressed as $f=\vee\{h: h\in {\mathscr A}(E),h\leq f\}$, then according to Proposition \ref{exchange}, every stable and fully order preserving operator $T:{\mathscr C}(E)\to {\mathscr C}(E)$ is fully determined by its action on ${\mathscr A}(E)$. Hence, we will analyze the behavior of the restriction of a stable and fully order preserving operator to ${\mathscr A}(E)$.

Next, we prove that each stable and fully order preserving operator maps any $L^0$-affine function to an $L^0$-affine function. Compared with the proof of \cite[Proposition 3]{IRS}, our proof must use the stable property of $T$.

\begin{proposition}
  If $T: \mathscr{C}(E) \rightarrow \mathscr{C}(E)$ is stable and fully order preserving, then we have: \\
(1). $T(h)\in {\mathscr A}(E)$ for all $h\in {\mathscr A}(E)$.\\
(2). If $T(f)\in {\mathscr A}(E)$ for some $f\in {\mathscr C}(E)$ then $f\in {\mathscr A}(E)$.
\end{proposition}

\begin{proof}
(1) Let $h=h_{u,\alpha}$. We will use Proposition \ref{AffCha} in order to establish $L^0$-affineness of $T(h)$. Take $h_{v,\delta}$ and $h_{v^\prime,\delta^\prime}$ in ${\mathscr A}(T(h))$, i.e. $h_{v,\delta}\leq T(h)$ and $h_{v^\prime,\delta^\prime}\leq T(h)$. Since $T$ is onto, there exist $g$ and $g^\prime$ in ${\mathscr C}(E)$ such that $h_{v,\delta}=T(g)$ and $h_{v^\prime,\delta^\prime}=T(g^\prime)$, so that $T(g)\leq T(h)$ and $T(g^\prime)\leq T(h)$. Since $T$ is fully order preserving, we get $g\leq h$ and $g^\prime\leq h$. By Proposition \ref{comparison}, there exist $\eta$ and $\eta^\prime$ in $L^0({\mathcal F})$ such that $g=h_{u,\eta}$ and $g^\prime=h_{u,\eta^\prime}$. Let $A=[\eta\geq \eta^\prime]$ and denote $\eta^{\prime\prime}={\tilde I}_A\eta+{\tilde I}_{A^c}\eta^\prime$, then $\eta\leq\eta^{\prime\prime}$ and $\eta^\prime\leq\eta^{\prime\prime}$. Denote $g^{\prime\prime}={\tilde I}_Ag+{\tilde I}_{A^c}g^\prime=h_{u,\eta^{\prime\prime}}$, then $g\leq g^{\prime\prime}$ and $g^\prime\leq g^{\prime\prime}$ and by the assumption that $T$ is stable, $T(g^{\prime\prime})={\tilde I}_AT(g)+{\tilde I}_{A^c}T(g^\prime)=h_{v^{\prime\prime},\delta^{\prime\prime}}$, where $v^{\prime\prime}={\tilde I}_Av+{\tilde I}_{A^c}v^\prime$ and $\delta^{\prime\prime}={\tilde I}_A\delta+{\tilde I}_{A^c}\delta^\prime$. Since $T$ is order preserving,
$h_{v,\delta}=T(g)\leq T(g^{\prime\prime})=h_{v^{\prime\prime},\delta^{\prime\prime}}$ and $h_{v^\prime,\delta^\prime}=T(g^\prime)\leq T(g^{\prime\prime})=h_{v^{\prime\prime},\delta^{\prime\prime}}$.
  It follows from Proposition \ref{comparison} that $v=v^\prime=v^{\prime\prime}$. We have proved that there exists a unique $v\in E^*$ such that $h_{v,\delta}\in {\mathscr A}(T(h))$ for some $\delta\in L^0({\mathcal F})$, and hence $T(h)$ is $L^0$-affine by Proposition \ref{AffCha}.

(2) By Proposition \ref{inverse}, $T$'s inverse $T^{-1}$ is also stable and fully order preserving. It follows that $T^{-1}$ also maps an $L^0$-affine function to an $L^0$-affine function, which implies the result.
\end{proof}

We have seen that any stable and fully order preserving operator maps ${\mathscr A}(E)$ to ${\mathscr A}(E)$. For a stable and fully order preserving operator $T: \mathscr{C}(E) \rightarrow \mathscr{C}(E)$, denote by $\hat T:{\mathscr A}(E)\to {\mathscr A}(E)$ the restriction of $T$ to ${\mathscr A}(E)$. Note that $\hat T$ inherits from $T$ the properties of being onto and fully order preserving on ${\mathscr A}(E)$.

For the following discussion, we recall some notions from \cite{WGL}. Let $E_1$ and $E_2$ be two $L^0({\mathcal F})$-modules, and $R: E_1\to E_2$ a mapping.
$R$ is said to be $L^0$-linear, if $R$ is a module homomorphism, namely $R(x+y)=R(x)+R(y)$ for any $x,y\in E_1$, and $R(\xi x)=\xi R(x)$ for any $\xi\in L^0({\mathcal F})$ and any $x\in E_1$; $R$ is said to be $L^0$-affine, if $R(\lambda x+(1-\lambda)y)=\lambda R(x)+(1-\lambda)R(y)$ for any $\lambda\in L^0({\mathcal F})$ and any $x,y\in E_1$, equivalently, $R(\cdot)-R(\theta)$ is $L^0$-linear; $R$ is said to be stable, if $R({\tilde I}_A x+{\tilde I}_{A^c} y)={\tilde I}_AR(x)+{\tilde I}_{A^c} R(y)$ for all $x,y\in E_1$ and $A\in {\mathcal F}$.

By associating $h_{u,\alpha}$ to the pair $(u,\alpha)$, we can identify ${\mathscr A}(E)$ with $E^*\times L^0({\mathcal F})$, and hence we take $\hat T$ as an operator acting on $E^*t\times L^0({\mathcal F})$ and write ${\hat T} (u,\alpha)$ instead of ${\hat T}(h_{u,\alpha})$. We will prove next that $\hat T$ is $L^0$-affine.

\begin{proposition}\label{hT}
$\hat T:E^*\times L^0({\mathcal F})\to E^*\times L^0({\mathcal F})$ is bijective and stable.
\end{proposition}
\begin{proof}
 It is obvious that $\hat T$ is a bijection. We only need to show that $\hat T$ is stable. For any $(u,\alpha)$ and $(v,\beta)$ in $E^*\times L^0({\mathcal F})$ and any $A$ in ${\mathcal F}$, since $$h_{{\tilde I}_A u+{\tilde I}_{A^c} v, {\tilde I}_A \alpha+{\tilde I}_{A^c} \beta}={\tilde I}_Ah_{u,\alpha}+{\tilde I}_{A^c}h_{v,\beta},$$ it follows from the stable property of $T$ that $$T(h_{{\tilde I}_A u+{\tilde I}_{A^c} v, {\tilde I}_A \alpha+{\tilde I}_{A^c} \beta})={\tilde I}_AT(h_{u,\alpha})+{\tilde I}_{A^c}T(h_{v,\beta}),$$ namely $$\hat T({\tilde I}_A u+{\tilde I}_{A^c} v, {\tilde I}_A \alpha+{\tilde I}_{A^c} \beta)={\tilde I}_A\hat T(u,\alpha)+{\tilde I}_{A^c}\hat(v,\beta),$$ thus $\hat T$ is stable.
\end{proof}

Lemma \ref{uni} below tells us that for the mapping ${\hat T}: (u,a)\mapsto (v,\beta)$, $v$ only depends on $u$.

\begin{lemma}\label{uni}
Suppose that ${\hat T}(u,\alpha_1)=(v_1,\beta_1)$ and ${\hat T}(u,\alpha_2)=(v_2,\beta_2)$, then we have $v_1=v_2$. Conversely, if ${\hat T}(u_1,\alpha_1)=(v,\beta_1)$ and ${\hat T}(u_2,\alpha_2)=(v,\beta_2)$, then we have $u_1=u_2$.
\end{lemma}
\begin{proof}
Arbitrarily choose an $\alpha\in L^0({\mathcal F})$ such that $\alpha\geq \alpha_1 \vee \alpha_2$ and suppose that ${\hat T}(u,\alpha)=(v,\beta)$. Since $h_{u,\alpha}\geq h_{u,\alpha_1}$ and $T$ is fully order preserving, we have $T(h_{u,\alpha})\geq T(h_{u,\alpha_1})$, namely $h_{v,\beta}\geq h_{v_1,\beta_1}$,
which implies $v_1=v$ by Proposition \ref{comparison}. Similarly we have $v_2=v$, thus $v_1=v_2$. By considering the inverse of $\hat T$ which is obvious $\hat {T^{-1}}$, the converse part immediately follows.
\end{proof}

\begin{proposition}\label{supp}
 Assume that ${\hat T}(u_1,\alpha_1)=(v_1,\beta_1)$ and ${\hat T}(u_2,\alpha_2)=(v_2,\beta_2)$. If $u_1-u_2$ has full support, then $v_1-v_2$ has full support.
\end{proposition}

\begin{proof}

We prove the conclusion by contradiction. Suppose there exists $A\in {\mathcal F}_+$ such that ${\tilde I}_A (v_1-v_2)=0$, equivalently, ${\tilde I}_Av_1={\tilde I}_Av_2$. Let $u={\tilde I}_Au_1+{\tilde I}_{A^c}u_2$ and $\alpha ={\tilde I}_A\alpha_1+{\tilde I}_{A^c}\alpha_2$, since $\hat T$ is stable, we have ${\hat T} (u, \alpha)={\tilde I}_A{\hat T}(u_1,\alpha_1)+{\tilde I}_{A^c}{\hat T}(u_2,\alpha_2)={\tilde I}_A(v_1,\beta_1)+{\tilde I}_{A^c}(v_2,\beta_2)=(v_2,{\tilde I}_A\beta_1+{\tilde I}_{A^c}\beta_2)$. It follows from Lemma \ref{uni} that $u=u_2$, then ${\tilde I}_Au={\tilde I}_Au_1={\tilde I}_Au_2$, which contradicts to the assumption that $u_1-u_2$ has full support.
\end{proof}

 Let us recall the notion of $L^0$-line segments \cite{WGL}: for any two elements $x$ and $y$ of an $L^0({\mathcal F})$-module, $[x,y]:=\{\lambda x+(1-\lambda)y: \lambda\in L^0_+({\mathcal F}), 0\leq \lambda\leq 1\}$, called the $L^0$-line segment linking $x$ and $y$.

 To establish the $L^0$-affineness of $\hat T$, we need clarify $\hat T$'s action on $L^0$-line segments. Please bare in mind that for any $h_1$ and $h_2$ in ${\mathscr A}(E)$, $h_1<h_2$ means $h_1\leq h_2$ and $h_1\neq h_2$.

\begin{proposition}\label{line-line}
$\hat T:E^*\times L^0({\mathcal F})\to E^*\times L^0({\mathcal F})$ maps any $L^0$-line segment to an $L^0$-line segment. Namely, $\hat T([z_1,z_2])=[\hat Tz_1,\hat Tz_2]$ for any $z_1$ and $z_2$ in $E^*\times L^0({\mathcal F})$.
\end{proposition}

\begin{proof}
 Observe that if we can prove $\hat T$ maps any $L^0$-line segment into an $L^0$-line segment, namely $\hat T([z_1,z_2])\subset [\hat Tz_1,\hat Tz_2]$ for any $z_1$ and $z_2$ in $E^*\times L^0({\mathcal F})$, by considering $\hat T^{-1}$ we can see that $\hat T$ would also map any $L^0$-line segment onto an $L^0$-line segment, namely, $\hat T([z_1,z_2])=[\hat Tz_1,\hat Tz_2]$ for any $z_1$ and $z_2$ in $E^*\times L^0({\mathcal F})$. Thus we only need to show that $\hat T$ maps any $L^0$-line segment into an $L^0$-line segment.

Fix two elements $(u_1,\alpha_1)$ and $(u_2,\alpha_2)$ in $E^*\times L^0({\mathcal F})$. Assume ${\hat T}(u_1,\alpha_1)=(v_1,\beta_1)$ and ${\hat T}(u_2,\alpha_2)=(v_2,\beta_2)$. Fix a $\lambda\in L^0({\mathcal F})$ with $0\leq\lambda\leq 1$, let $u=\lambda u_1+(1-\lambda)u_2$ and $\alpha=\lambda\alpha_1+(1-\lambda)\alpha_2$, and suppose ${\hat T}(u,\alpha)=(v,\beta)$, we shall show that there exists $\mu\in L^0({\mathcal F})$ with $0\leq\mu\leq 1$ such that $v=\mu v_1+(1-\mu)v_2$ and $\beta=\mu \beta_1+(1-\mu)\beta_2$. The proof is divided into 3 steps.

{\it Step 1}. If $u_1=u_2$, then the conclusion holds.

From the construction we see that $u=u_1=u_2$, then it follows from Lemma \ref{uni} that $v=v_1=v_2$. We only need to find $\mu\in L^0({\mathcal F})$ with $0\leq\mu\leq 1$ such that $\beta=\mu \beta_1+(1-\mu)\beta_2$.

By Proposition \ref{exchange} we obtain
$${\hat T}(u,\alpha_1\vee \alpha_2)=(v, \beta_1\vee\beta_2).$$
Since $\hat T$ is fully order preserving, we have
$${\hat T}(u,\alpha_1\wedge \alpha_2)\leq ({\hat T}(u,\alpha_1)\wedge {\hat T}(u,\alpha_2))=(v,\beta_1\wedge\beta_2),$$
similarly, by considering the fully order preserving operator $\hat T^{-1}$ we get
$$\hat T^{-1}(v,\beta_1\wedge\beta_2)\leq (u,\alpha_1\wedge \alpha_2),$$
then it follows that
$$(v,\beta_1\wedge\beta_2)\leq \hat T(u,\alpha_1\wedge \alpha_2),$$
 thus we obtain
 $${\hat T}(u,\alpha_1\wedge \alpha_2)=(v, \beta_1\wedge\beta_2).$$
  Note that $(\alpha_1\wedge \alpha_2)\leq\alpha\leq (\alpha_1\vee \alpha_2)$,
  we thus deduce
  $$(\beta_1\wedge\beta_2)\leq \beta\leq (\beta_1\vee\beta_2).$$
  Then there must exist $\mu_0\in L^0({\mathcal F})$ with $0\leq\mu_0\leq 1$ such that $$\beta=\mu_0 (\beta_1\wedge\beta_2)+(1-\mu_0)(\beta_1\vee\beta_2).$$
  Denote $A=[\beta_1\geq \beta_2]$, then $(\beta_1\wedge\beta_2)=I_{A^c}\beta_1+I_{A}\beta_2$ and $(\beta_1\vee\beta_2)=I_{A}\beta_1+I_{A^c}\beta_2$. Let $\mu=\mu_0I_{A^c}+(1-\mu_0)I_{A}$, we see that $0\leq \mu\leq 1$ and $\beta=\mu \beta_1+(1-\mu)\beta_2$.

{\it Step 2}. If $u_1-u_2$ has full support, then the conclusion holds.

Note by Proposition \ref{supp}, $v_1-v_2$ also has full support.

Since $$h_{u,\alpha}=\lambda h_{u_1,\alpha_1}+(1-\lambda)h_{u_2,\alpha_2}\leq h_{u_1,\alpha_1}\vee h_{u_2,\alpha_2},$$
and $T$ is fully order preserving, we obtain
$$T(h_{u,\alpha})\leq T(h_{u_1,\alpha_1}\vee h_{u_2,\alpha_2}).$$
From Proposition \ref{exchange}, $T(h_{u_1,\alpha_1}\vee h_{u_2,\alpha_2})=T(h_{u_1,\alpha_1})\vee T(h_{u_2,\alpha_2})$, thus we get $h_{v,\beta}\leq h_{v_1,\beta_1}\vee h_{v_2,\beta_2}$, namely
\begin{equation}\label{P308:1}
\langle v,x\rangle+\beta\leq (\langle v_1,x\rangle+\beta_1)\vee (\langle v_2,x\rangle+\beta_2), \forall x\in E,
\end{equation}
by Proposition \ref{subdiff},
\begin{equation}\label{P308:2}
\langle v,x\rangle\leq (\langle v_1,x\rangle\vee \langle v_2,x\rangle), \forall x\in E,
\end{equation}
and there exists $\mu\in L^0({\mathcal F})$ with $0\leq\mu\leq 1$ such that $v=\mu v_1+(1-\mu)v_2$. It remains to show that $\beta=\mu \beta_1+(1-\mu)\beta_2$ holds for this $\mu$. To this end, we first show that there exists an $x_0\in E$ such that
\begin{equation}\label{P308:3}
\langle v,x_0\rangle+\beta=\langle v_1,x_0\rangle+\beta_1=\langle v_2,x_0\rangle+\beta_2.
\end{equation}

Since $v_1-v_2$ has full support, it follows from Proposition \ref{support}(2) that there exists an $x_0\in E$ such that $\langle v_1-v_2, x_0\rangle=\beta_2-\beta_1$, equivalently, $\langle v_1,x_0\rangle+\beta_1=\langle v_2,x_0\rangle+\beta_2$. Let $\epsilon=(\langle v_1,x_0\rangle+\beta_1)-(\langle v,x_0\rangle+\beta)$, from \eqref{P308:1} we see that $\epsilon\geq 0$. To show that \eqref{P308:3} holds for the $x_0$, it suffices to show that $\epsilon=0$.

For any $x\in E$, it follows from \eqref{P308:2} that
\begin{equation*}
 \langle v,x-x_0\rangle\leq (\langle v_1,x-x_0\rangle)\vee (\langle v_2,x-x_0\rangle),
\end{equation*}
thus, \begin{eqnarray*}
       & &\langle v,x\rangle+\beta+\epsilon\\
       &=&\langle v,x-x_0\rangle+\langle v,x_0\rangle+\beta+\epsilon \\
       &=& \langle v,x-x_0\rangle+(\langle v_1,x_0\rangle+\beta_1)\\
       &\leq& (\langle v_1,x-x_0\rangle)\vee (\langle v_2,x-x_0\rangle)+(\langle v_1,x_0\rangle+\beta_1)\\
       &=&(\langle v_1,x-x_0\rangle+\langle v_1,x_0\rangle+\beta_1)\vee (\langle v_2,x-x_0\rangle+\langle v_1,x_0\rangle+\beta_1)\\
       &=&(\langle v_1,x-x_0\rangle+\langle v_1,x_0\rangle+\beta_1)\vee (\langle v_2,x-x_0\rangle+\langle v_2,x_0\rangle+\beta_2)\\
       &=&(\langle v_1,x\rangle+\beta_1) \vee (\langle v_2,x\rangle+\beta_2),
\end{eqnarray*}
namely we get
\begin{equation}\label{P308:4}
h_{v,\beta+\epsilon}\leq h_{v_1,\beta_1}\vee h_{v_2,\beta_2}.
\end{equation}

We can now show $\epsilon=0$ by contradiction.

If $\epsilon >0$, then $h_{v,\beta}<h_{v,\beta+\epsilon}$. Due to our assumption that $T(h_{u,\alpha})=h_{v,\beta}$, from Lemma \ref{uni} we can assume that $T(h_{u,\alpha^\prime})=h_{v,\beta+\epsilon}$. Since $T$ is fully order preserving, we must have $\alpha^\prime>\alpha$, besides, noting that $T(h_{u_1,\alpha_1})=h_{v_1,\beta_1}$ and $T(h_{u_2,\alpha_2})=h_{v_2,\beta_2}$, from \eqref{P308:4} we have
\begin{equation}\label{P308:5}
h_{u,\alpha^\prime}\leq h_{u_1,\alpha_1}\vee h_{u_2,\alpha_2}.
\end{equation}
 Using the assumption that $u_1-u_2$ has full support, by Proposition \ref{support}(2) there exists some $x_1\in E$ such that $\langle u_1-u_2, x_1\rangle=\alpha_2-\alpha_1$, equivalently, $\langle u_1,x_1\rangle+\alpha_1=\langle u_2,x_1\rangle+\alpha_2$, namely $h_{u_1,\alpha_1}(x_1)=h_{u_2,\alpha_2}(x_1)$. Then using $h_{u,\alpha}=\lambda h_{u_1,\alpha_1}+(1-\lambda)h_{u_2,\alpha_2}$ we get
 \begin{eqnarray*}
       h_{u,\alpha^\prime}(x_1)&=&h_{u,\alpha}(x_1)+(\alpha^\prime-\alpha)\\
                               &=&\lambda h_{u_1,\alpha_1}(x_1)+(1-\lambda)h_{u_2,\alpha_2}(x_1)+(\alpha^\prime-\alpha)\\
                               &=& h_{u_1,\alpha_1}(x_1)+(\alpha^\prime-\alpha)\\
                               &>& h_{u_1,\alpha_1}(x_1)\\
                               &=& h_{u_1,\alpha_1}(x_1)\vee h_{u_2,\alpha_2}(x_1),
\end{eqnarray*}
this contradicts to \eqref{P308:5}.

Now that we have shown that \eqref{P308:3} holds for some $x_0$, then for this $x_0$,
 \begin{eqnarray*}
     \langle v,x_0\rangle+\beta&=&\mu(\langle v_1,x_0\rangle+\beta_1)+(1-\mu)(\langle v_2,x_0\rangle+\beta_2)\\
                               &=&\mu\langle v_1,x_0\rangle+(1-\mu)\langle v_2,x_0\rangle+\mu\beta_1+(1-\mu)\beta_2\\
                               &=&\langle \mu v_1+(1-\mu)v_2,x_0\rangle+\mu\beta_1+(1-\mu)\beta_2\\
                               &=&\langle v,x_0\rangle+\mu\beta_1+(1-\mu)\beta_2,
\end{eqnarray*}
thus we have $\beta=\mu \beta_1+(1-\mu)\beta_2$.

{\it Step 3}. Generally, the conclusion holds.

 Denote $A=[\|u_1-u_2\|=0]$. From Proposition \ref{support}(1), there exists $u_0\in E^*$ with $\|u_0\|=1$. Let $w=I_Au_0+u_1$ and $u_3=I_A w+I_{A^c}u_2$, then $\|u_3-u_1\|=I_A\|u_0\|+I_{A^c}\|u_1-u_2\| \neq 0$ on $\Omega$, namely $u_3-u_1$ has full support.

Now according to {\it Step 1}, there exists $\mu_1\in L^0({\mathcal F})$ with $0\leq \mu_1\leq 1$ such that
$${\hat T}(u_1, \alpha)=\mu_1{\hat T}(u_1, \alpha_1)+(1-\mu_1){\hat T}(u_1, \alpha_2),$$
according to {\it Step 2}, there exists $\mu_2\in L^0({\mathcal F})$ with $0\leq \mu_2\leq 1$ such that
$${\hat T}(\lambda u_1+(1-\lambda)u_3, \alpha)=\mu_2{\hat T}(u_1, \alpha_1)+(1-\mu_2){\hat T}(u_3, \alpha_2).$$
Since
 \begin{eqnarray*}
   & &I_Au_1+I_{A^c}(\lambda u_1+(1-\lambda)u_3)\\
   &=&I_A\lambda u_1+I_A(1-\lambda) u_1+I_{A^c}\lambda u_1+I_{A^c}(1-\lambda)u_3\\
   &=&I_A\lambda u_1+I_A(1-\lambda) u_2+I_{A^c}\lambda u_1+I_{A^c}(1-\lambda)u_2\\
   &=&\lambda u_1+(1-\lambda)u_2\\
   &=& u,
\end{eqnarray*}
 by the stable property of $\hat T$, we get
  \begin{eqnarray*}
   & &{\hat T}(u, \alpha)\\
   &=&I_A{\hat T}(u_1, \alpha)+I_{A^c}{\hat T}(\lambda u_1+(1-\lambda)u_3, \alpha)\\
   &=&I_A(\mu_1{\hat T}(u_1, \alpha_1)+(1-\mu_1){\hat T}(u_1, \alpha_2))+I_{A^c}(\mu_2{\hat T}(u_1, \alpha_1)+(1-\mu_2){\hat T}(u_3, \alpha_2))\\
   &=&(I_A\mu_1+I_{A^c}\mu_2){\hat T}(u_1, \alpha_1)+I_A(1-\mu_1){\hat T}(u_2, \alpha_2)+I_{A^c}(1-\mu_2){\hat T}(u_2, \alpha_2)\\
   &=&(I_A\mu_1+I_{A^c}\mu_2){\hat T}(u_1, \alpha_1)+[I_A(1-\mu_1)+I_{A^c}(1-\mu_2)]{\hat T}(u_2, \alpha_2),
\end{eqnarray*}
 if we set $\mu=I_A\mu_1+I_{A^c}\mu_2$, then $\mu\in L^0({\mathcal F})$ with $0\leq \mu\leq 1$ and
  $${\hat T}(u, \alpha)=\mu{\hat T}(u_1, \alpha_1)+(1-\mu){\hat T}(u_2, \alpha_2).$$
This finally completes the proof.
\end{proof}

Now we can give an explicit form of $\hat T$.

\begin{proposition}\label{Affrep}
$\hat T:E^*\times L^0({\mathcal F})\to E^*\times L^0({\mathcal F})$ is invertible and $L^0$-affine. Precisely, there exist $w\in E^*, \tau\in L^0_{++}({\mathcal F}), \beta\in L^0({\mathcal F})$, a module automorphism $D: E^*\to E^*$ and an $L^0$-linear function $d: E^*\to L^0({\mathcal F})$ such that for each $(u,\alpha)\in E^*\times L^0({\mathcal F})$,
$${\hat T}(u,\alpha)=(Du+w,d(u)+\tau \alpha+\beta).$$
\end{proposition}
\begin{proof}
  According to Proposition \ref{hT}, $\hat T$ is invertible. By Proposition \ref{support}(1), there exists $u_0\in E^*$ such that $\|u_0\|=1$. If $\xi$ and $\eta$ belong to $L^0({\mathcal F})$ and $\xi(u_0,0)+\eta(1,0)=(
  \xi u_0,\eta)=(\theta,0)$, then clearly $\xi=\eta=0$, thus $(u_0,0)$ and $(0,1)$ are $L^0$-independent in $E^*\times L^0({\mathcal F})$. By Proposition \ref{line-line}, $\hat T:E^*\times L^0({\mathcal F})\to E^*\times L^0({\mathcal F})$ maps any $L^0$-line segment to an $L^0$-line segment, invoking the fundamental theorem of affine geometry in regular $L^0$-modules \cite[Theorem 1.1]{WGL}, $\hat T$ must be $L^0$-affine. In view of Lemma \ref{uni}, we can denote ${\hat T}(u,\alpha)=(y(u),\gamma(u,\alpha))$ for each $(u,\alpha)\in E^*\times L^0({\mathcal F})$.
Then both $y: E^*\to E^*$ and $\gamma: E^*\times L^0({\mathcal F}) \to L^0({\mathcal F})$ are $L^0$-affine, namely there exist $w\in E^*, \tau\in L^0({\mathcal F}), \beta\in L^0({\mathcal F})$, a module homomorphism $D: E^*\to E^*$ and an $L^0$-linear function $d: E^*\to L^0({\mathcal F})$ such that $y(u)=Du+w, \gamma(u,\alpha)=d(u)+\tau \alpha+\beta$. Since $\hat T$ is bijective and preserve the order of the second variable, we can see that $D$ must be invertible and $\tau\in L^0_{++}({\mathcal F})$.
\end{proof}

\begin{remark}
In Iusem, Reem and Svaiter \cite{IRS}, the affineness of $\hat T: X^*\times \mathbb R\to X^*\times \mathbb R$ is established separately for the mapping $y: X^*\to X^*$ (see \cite[Corollary 4]{IRS}) and the mapping $\gamma: X^*\times \mathbb R\to \mathbb R$ (see \cite[Proposition 7]{IRS}). Since the affineness of $y$ relies on the fundamental theorem of affine geometry, Iusem, Reem and Svaiter had to impose the assumption that the Banach space $X$ in their main results, Theorem 1 and Theorem 2, has dimension not less than 2. As a comparison, we establish the $L^0$-affineness of $\hat T$ as a whole so that we do not need an additional assumption that $E$ contains a free submodule with rank 2.
\end{remark}

We proceed to verify the a.s. boundedness of $D$ and $d$.

\begin{proposition}\label{asbounded}
Both $D$ and $d$ in Proposition \ref{Affrep} are a.s. bounded and thus continuous. Specially, $d\in E^{**}$.
\end{proposition}
\begin{proof}
 We first prove $d$ is a.s. bounded. Since $g=\|\cdot\|\in \mathscr{C}(E)$, we denote $g_1=T(g)$. Let $B$ be the closed unit ball of $E^*$, i.e. $B=\{u\in E^*: \|u\|\leq 1\}$. Since $\|x\|=\vee\{h_{u,0}(x)=\langle u, x\rangle: u\in B\}$ for every $x\in E$, we have
\begin{eqnarray*}
 g_1(x)&=&[T(g)](x)\\
       &=&\vee\{T(h_{u,0})(x):u\in B\}\\
       &=& \vee\{h_{\hat T(u,0)}(x):u\in B\}\\
       &=& \vee\{\langle Du+w, x\rangle+d(u)+\beta:u\in B\}.
\end{eqnarray*}
Therefore, $g_1(0)=\vee \{d(u)+\beta: u\in B\}=\vee \{d(u): u\in B\}+\beta$.
 By Proposition \ref{valuefinite}, we have $\vee\{d(u):u\in B\}=g_1(0)-\beta\in L^0({\mathcal F})$, thus $d$ is a.s. bounded with $\|d\|=g_1(0)-\beta$.

We then prove $D$ is a.s. bounded.  By Proposition \ref{valuefinite},
$$\vee\{\langle Du+w, x\rangle+d(u)+\beta:u\in B\}=g_1(x)\in L^0({\mathcal F}), \forall x\in E,$$
thus
$$\langle Du, x\rangle\leq g_1(x)-\langle w, x\rangle-d(u)-\beta\leq g_1(x)-\langle w, x\rangle+\|d\|-\beta$$ for every $u\in B$ and $x\in E$.
Then $$\vee\{|\langle Du, x\rangle|: u\in B\}=\vee\{\langle Du, x\rangle: u\in B\}\leq g_1(x)-\langle w, x\rangle+\|d\|-\beta\in L^0({\mathcal F})$$ for every $x\in E$. Using uniform boundedness principle in RN modules (see \cite[Theorem 29]{Guo-Modulehomo}), we obtain that $\{Du:u\in B\}$ is a.s bounded, therefore $D$ is a.s. bounded. This completes the proof.
\end{proof}

Now we can give the proof of Theorem \ref{OP}.

{\bf Proof of Theorem \ref{OP}}. Sufficiency can be easily shown by straightforward verification. We only prove the necessity.

From Corollay \ref{Affsup}, for any $f\in \mathscr{C}(E)$ we have that $f=\vee\{h_{u,\alpha}: h_{u,\alpha}\leq f\}$.
Then according to Proposition \ref{exchange},
$T(f)=\vee\{Th_{u,\alpha}: h_{u,\alpha}\leq f\}$.
 By Proposition \ref{Affrep} and Proposition \ref{asbounded}, there exist $w\in E^*$, $d\in E^{**}$, $\tau\in L^0_{++}({\mathcal F}), \beta\in L^0({\mathcal F})$ and a continuous module automorphism $D$ of $E^*$ such that
 $${\hat T}(u,\alpha)=(Du+w,\langle d,u\rangle+\tau \alpha+\beta),$$
thus for each $x\in E$,
\begin{eqnarray*}
  [T(f)](x)&=&\vee \{\langle Du+w,x\rangle+\langle d,u\rangle+\tau \alpha+\beta: h_{u,\alpha}\leq f \}\\
           &=& \vee\{\langle Du,x\rangle+\langle d,u\rangle+\tau \alpha: h_{u,\alpha}\leq f \}+\langle w,x\rangle+\beta\\
           &=& \tau \vee\{\langle Cu,x\rangle+\langle c,u\rangle+\alpha: h_{u,\alpha}\leq f\}+\langle w,x\rangle+\beta \\
           &=& \tau \vee\{\langle C^*x+c,u\rangle+\alpha: h_{u,\alpha}\leq f \}+\langle w,x\rangle+\beta,
\end{eqnarray*}
where $C=\tau^{-1}D, c=\tau^{-1}d$ and $C^*: E^{**}\to E^{**}$ is the conjugate of $C$, so that $C^* x$ must be understood through the natural embedding $E\subset E^{**}$. It follows from Proposition \ref{bi-conjugate} that
\begin{equation}\label{T1}
  [T(f)](x)=\tau f^{**}(C^* x+c)+\langle w,x\rangle+\beta,  \quad\forall x\in E.
\end{equation}

Define $Z\subset E^{**}$ as $Z=\{C^* x+c: x\in E\}$. We will prove that $Z=E$.

We first show that $E\subset Z$. Assume that there exists $\tilde x\in E\setminus Z$. Consider the indicator function $\delta_{\tilde x}: E\to \bar L^0({\mathcal F})$. From Lemma \ref{Indicator}, $\delta_{\tilde x}\in \mathscr{C}(E)$ and $\delta^{**}_{\tilde x}(x^{**})=\delta_{\tilde x}(x^{**})$ for every $x^{**}\in E^{**}$, where $\tilde x$ is seen as an element in $E^{**}$.
It follows from \eqref{T1} that $$[T(\delta_{\tilde x})](x)=\tau \delta_{\tilde x}(C^* x+c)+\langle w,x\rangle+\beta, \quad\forall x\in E.$$
However, since $\tilde x\notin Z$, we have $C^* x+c\neq \tilde x$ for all $x\in E$, thus $$\delta_{\tilde x}^{**}(C^* x+c)=\infty  \text{~on~} [C^* x+c\neq \tilde x]\in {\mathcal F}_+$$ for all $x\in E$, which implies that $dom (T(\delta_{\tilde x}))=\emptyset$. This is a contradiction.

We then show that $Z\subset E$. Suppose now that there exists $\check x\in Z\setminus E$. Since $\check x\in Z$, there exists $x^\prime\in E$ such that $C^* x^\prime+c=\check x$. Consider the indicator function $\delta_{x^\prime}: E\to \bar L^0({\mathcal F})$.
Since $T$ is a bijection, there exists $g\in \mathscr{C}(E)$ such that $T(g)=\delta_{x^\prime}$. Then it follows from \eqref{T1} that
 $$\delta_{x^\prime}(x)=\tau g^{**}(C^* x+c)+\langle w,x\rangle+\beta,  \quad\forall x\in E.$$
 Now that $dom(\delta_{x^\prime})=\{x^\prime\}$, we thus have $Z \cap dom (g^{**})=\{C^* x^\prime+c\}=\{\check x\}$. Since $E\subset Z$ and $\check x\notin E$, we obtain that $E\cap dom (g^{**})=\emptyset$. From the random Fenchel-Moreau duality theorem we see that $g^{**}|_E=g$, thus $dom(g)=E\cap dom (g^{**})=\emptyset$. This is a contradiction.

We have completed the proof of $E=Z$.
We observe now that, since $C^* x+c\in E$ for all $x\in E$, by taking $x=0$, we obtain $c\in E$. As a consequence, $H:=C^*|_E$, the restriction of $C^*$ to $E$, is a continuous module automorphism of $E$.

Using $f^{**}|_E=f$, we obtain that
$$ [T(f)](x)=\tau f^{**}(C^* x+c)+\langle w,x\rangle+\beta=\tau f(Hx+c)+\langle w,x\rangle+\beta, \forall x\in E.$$
This completes the proof. \hfill$\square$

Then we give the proof of Corollary \ref{involution}

{\bf Proof of Corollary \ref{involution}} This is just a copy of the proof of \cite[Corollary 6]{IRS}. \hfill$\square$

As a complement of Theorem \ref{OP}, we give an example showing that a fully order preserving operator $T: \mathscr{C}(E)\to \mathscr{C}(E)$ is not necessarily stable.

\begin{example}
Let the probability space $(\Omega,{\mathcal F}, P)$ be $([0,1),{\mathcal B}, m)$, where ${\mathcal B}$ is the Borel $\sigma$-algebra of $[0,1)$ and $m$ the Lebesgue measure. Set $(E,\|\cdot\|)=(L^0({\mathcal F}),|\cdot|)$. Define a mapping $\theta: [0,1)\to [0,1)$ as $\theta(\omega)=\omega+\frac{1}{2}$ for $\omega\in [0,\frac{1}{2})$ and $\theta(\omega)=\omega-\frac{1}{2}$ for $\omega\in [\frac{1}{2},1)$. Then $\theta$ is an involution, namely $\theta$ is a bijection with inverse $\theta^{-1}=\theta$. Since $\theta$ is obviously measure preserving, $\theta$ is an automorphism of $([0,1),{\mathcal B}, m)$. $\theta$ induces an involution $\sigma:\bar L^0({\mathcal F})\to \bar L^0({\mathcal F})$ by sending each $x\in \bar L^0({\mathcal F})$ to the equivalence class of $x^0(\theta(\cdot))$, where $x^0(\cdot)$ is a representative of $x$. It is easy to check that $\sigma$ is fully order preserving: for every $x,y\in \bar L^0({\mathcal F})$, $x\leq y$ iff $\sigma(x)\leq \sigma(y)$.

Now for every $f\in \mathscr{C}(E)$, define $T(f): E\to \bar L^0({\mathcal F})$ by $[T(f)](x)=\sigma[f(\sigma (x))], \forall x\in E$. We then check that $T(f) \in {\mathscr C}(E)$.

(1). For any $x\in E$, it is obvious that $[T(f)](x)>-\infty$ on $\Omega$.

(2). For any $x\in dom(f)$, we see that $\sigma (x)\in dom[T(f)]$, thus $dom[T(f)]$ is nonempty.

(3). For any $x,y\in E$ and $\lambda\in L^0({\mathcal F})$ with $0\leq \lambda\leq 1$, we have $f(\lambda x+(1-\lambda)y)\leq \lambda f(x)+(1-\lambda)f(y)$ since $f$ is $L^0$-convex, as a result
\begin{eqnarray*}
  [T(f)](\lambda x+(1-\lambda)y)&=&\sigma (f[\sigma(\lambda x+(1-\lambda)y)])\\
           &=& \sigma(f[\sigma(\lambda) \sigma(x)+\sigma(1-\lambda)\sigma(y))\\
           &\leq& \sigma[\sigma(\lambda) f(\sigma (x))+\sigma(1-\lambda)f(\sigma(y))] \\
           &=& \lambda \sigma (f(\sigma (x)))+(1-\lambda)\sigma(f(\sigma (y))\\
           &=&\lambda [T(f)](x)+(1-\lambda)[T(f)](y),
\end{eqnarray*}
 thus $T(f)$ is $L^0$-convex.

(4) Since $epi(f)$ is a closed subset of $E\times L^0({\mathcal F})$ and $$epi(T(f))=\{(\sigma(x),\sigma(r)):(x,r)\in epi(f)\},$$ we see that $epi(T(f))$ is also a closed subset of $E\times L^0({\mathcal F})$.

After this verification, by $f\mapsto Tf$ we obtain an operator $T: \mathscr{C}(E)\to \mathscr{C}(E)$. It is easy to see that $T$ is order preserving and $T(Tf)=f$ for every $f\in \mathscr{C}(E)$ since $\sigma^2$ is the identity. Thus $T$ is an order preserving involution.

We claim that $T$ is not stable. In fact, let $A=[0,\frac{1}{2})$ and $B=A^c=[\frac{1}{2},1)$, then $\theta(A)=B$, and $\theta(B)=A$, equivalently, $\sigma(\tilde I_A)=\tilde I_B$ and $\sigma(\tilde I_B)=\tilde I_A$. Choose $f_0\in \mathscr{C}(E)$ as $f_0(x)=x,\forall x\in E$. Then by construction,
$$[T(f_0)](x)=\sigma [f_0(\sigma (x))]=\sigma (\sigma (x))=x, \forall x\in E,$$
 namely $[T(f_0)]=f_0$, thus
 $$[\tilde I_A T(f_0)](x)=\tilde I_A f_0(x)=\tilde I_A x, \forall x\in E,$$
 and $$[T(\tilde I_A f_0)](x)=\sigma [\tilde I_A\sigma (x)]=\sigma(\tilde I_A)\sigma (\sigma (x))=\tilde I_B x, \forall x\in E.$$
 We see that $\tilde I_A T(f_0)\neq T(\tilde I_A f_0)$.
 Noting that $T(0)=0$, thus $$T(\tilde I_A f_0+\tilde I_B \cdot 0)=T(\tilde I_A f_0)\neq \tilde I_A T(f_0)=\tilde I_A T(f_0)+\tilde I_B T(0),$$ which implies that $T$ is not stable.
\end{example}

\section{The order reversing case}

This section gives the proof of Theorem \ref{OR}.

Proposition \ref{conjugateisomorphism} below is indeed implied in the proof of random Fenchel-Moreau duality theorem \cite[Theorem 5.1]{GZZ-RCA1}(see also \cite{GZWY}).

\begin{proposition}\label{conjugateisomorphism}
The random conjugate transform $f\mapsto f^*$ is a bijection from $\mathscr{C}(E)$ to $\mathscr{C}_{w^*}(E^*)$.
\end{proposition}

We give the proof of Theorem \ref{OR} as follows.

{\bf Proof of Theorem \ref{OR}}. Sufficiency can be easily shown by straightforward verification. We only prove the necessity.

 For an operator $S:\mathscr{C}(E)\to \mathscr{C}_{w^*}(E^*)$, it follows from Proposition \ref{conjugateisomorphism} that there exists a unique $g\in \mathscr{C}(E)$ for each $f\in \mathscr{C}(E)$ such that $S(f)=g^*$. We can define an operator $T: \mathscr{C}(E)\mapsto \mathscr{C}(E)$ by $T(f)=g$ such that $S(f)=g^*$ for each $f\in \mathscr{C}(E)$. If $S$ is a stable fully order reversing operator , then it is easily verified that $T$ is a stable fully order preserving operator. Thus according to Theorem \ref{OP}, there exist a continuous module automorphism $H_1: E\to E$, and $c\in E$, $w \in E^*$, and $\tau \in L^0_{++}({\mathcal F}), \rho_1\in L^0({\mathcal F})$ such that
 $$[T(f)](x)=\tau f(H_1x+c)+\langle w,x\rangle+\rho_1, \quad\forall f\in \mathscr{C}(E), x\in E.$$
Consequently, for each $u\in E^*$,
\begin{eqnarray*}
  & &[S(f)](u)\\
  &=&[T(f)]^*(u)\\
  &=&\vee \{\langle u,x\rangle-\tau f(H_1x+c)-\langle w,x\rangle-\rho_1: x\in E \}\\
  &=&\vee\{\langle u-w,H^{-1}_1(z-c)\rangle-\tau f(z)-\rho_1: z\in E\}\\
  &=&\vee\{\langle u-w,H^{-1}_1z\rangle-\tau f(z):z\in E\}-\langle u-w,H^{-1}_1c\rangle-\rho_1 \\
  &=&\tau\vee\{\langle u-w,Hz\rangle-f(z):z\in E\}+\langle u,-H^{-1}_1c\rangle+\langle w,H^{-1}_1c\rangle-\rho_1 \\
  &=& \tau f^*(H^* u+v)+\langle u,y\rangle+\rho,
\end{eqnarray*}
where $H=\frac{1}{\tau_1}H^{-1}_1, v=-H^* w, y=-H^{-1}_1c$ and $\rho=\langle w,H^{-1}_1c\rangle-\rho_1$.

We see that $H: E\to E$ is a continuous module automorphism, $v \in E^*$, $y\in E$ and $\rho\in L^0({\mathcal F})$.
\hfill$\square$

\section*{Acknowledgements}
 The first author was supported by the Natural Science Foundation of China (Grant No.11701531) and the Fundamental Research Funds for the Central Universities, China University of Geosciences (Wuhan) (No. CUGL170820). The second author is supported by the Natural Science Foundation of China(Grant No.11971483). The third author was supported by the Natural Science Foundation of China(Grant No.11501580).

\end{document}